\documentclass[sn-mathphys,Numbered]{sn-jnl}

\usepackage{graphicx}%
\usepackage{multirow}%
\usepackage{amsmath,amssymb,amsfonts}%
\usepackage{amsthm}%
\usepackage{mathrsfs}%
\usepackage[title]{appendix}%
\usepackage{xcolor}%
\usepackage{textcomp}%
\usepackage{manyfoot}%
\usepackage{booktabs}%
\usepackage{algorithm}%
\usepackage{algorithmicx}%
\usepackage{algpseudocode}%
\usepackage{listings}%

\newtheorem{theorem}{Theorem}
\newtheorem{proposition}[theorem]{Proposition}%

\begin{document}

\title[Article Title]{The Pauli group as Galois group of irreducible\newline
pure polynomials over $\mathbb{Q}$}


\author[1]{\fnm{Harald} \sur{Borner}}\email{harald.mtac.borner@gmail.com}

\author[2]{\fnm{Falko} \sur{Lorenz}}\email{lorenz.math@uni-muenster.de}

\affil[1]{\orgname{Bergblick College}, \orgaddress{\street{Im Talblick 9}, \city{Gaiberg}, \postcode{69251},  \country{Germany}}}

\affil[2]{\orgdiv{FB Mathematik}, \orgname{Univ. Münster}, \orgaddress{\street{Waldweg 17}, \city{Münster}, \postcode{48163},  \country{Germany}}}


\abstract{We realize the Pauli group $P$ as Galois group of polynomials over the rational numbers. It is shown by construction that each \textit{pure} polynomial in the infinite family
of the form $X^8+k^2$ for $k\neq \lambda^2, 2\lambda^2; k,\lambda \in \mathbb{Q}^*$ has Galois group $P$ over $\mathbb{Q}$. This form is also proven to be a necessary condition for a realization of $P$ by any \textit{pure} polynomial of degree 8. It automatically provides a realization of the quaternion group $Q_8$ by pure polynomials over a quadratic extension of $\mathbb{Q}$, whereas it is shown to be impossible to realize $Q_8$ over $\mathbb{Q}$ by any pure polynomial of degree 8. We link the results to Witt's criterion for embedding a biquadratic extension into a normal extension with Galois group $Q_8$ via ternary quadratic forms.  
This provides for a connection to the known realizability criteria for embeddings of $E_8=C_2^3$-extensions into one with the Pauli group as Galois group, and the interesting subtleties therein, via the equivalence of certain  quaternion algebras. We thus show how to exactly extend Witt's result of 1936 to an embedding of $E_8$ into (and realization of $Q_8$ within) the Pauli group.}


%




\maketitle
\newpage
\section{Introduction}\label{sec1}

Over the last century, the order-16 Pauli group $P$, structurally presented by a certain semidirect product 
$Q_8 \rtimes C_2$, or by the unique central product $Q_8 \circ C_4$, appeared in an increasingly wide variety of contexts, ranging from spin quantum mechanics and field theory (where it is  isomorphic to the $CPT$ symmetry group of the Dirac equation \cite{bib20}), to quantum information theory \cite{bib1}, and even into mathematical music theory
\cite{bib2}. 
For instance, it plays a major role in the theory of error-correcting codes via the Gottesman-Knill theorem. Generally, logical gates for multiple qubits can be modelled using tensor products of operators in the Pauli group \cite{bib3}\cite{bib13}. 

In abstract algebra, it also plays a unique role, since e.g. it was shown by Minác and Smith \cite{bib15} that the rigidity of fields (with at least 8 square classes) can be characterized  by the condition that the Pauli group does \textit{not}  occur as a Galois group of extensions over them, where a rigid, or C-field, is defined to be a field whose Witt ring is isomorphic to a group ring $\mathbb{Z}/n\mathbb{Z}[G]$, for some $n \ge 0$ and some group $G$ of involutions.

Correspondingly, many attributes of the Pauli group are well-known and widely shared among experts in algebra and physics. However, this does not mean that no surprising aspects remain to be uncovered. We shall address here one of these aspects, pertaining to its role as a Galois group over $\mathbb{Q}$, and we shall add another unique algebraic characterisation of the Pauli group (at the end of chapter 4).

While it is known, due to the results of Scholz, Shafarevich et al \cite{bib4}, that some polynomial in $\mathbb{Z}[X] \subset \mathbb{Q}[X]$ can be found that generates a splitting field $E$ as an extension over $\mathbb{Q}$, with the solvable finite group $P$ as its Galois group, it is far from obvious that one (or more) such polynomial can be discovered with a certain "simple" structure, eg. of minimal degree and limited to a minimum number of monomials.

For the Pauli group, computer algebra codes typically generate solutions to the first problem, like $x^8-2x^7-7x^6+16x^5+4x^4-18x^3+2x^2+4x-1$ \cite{bib5}. Furthermore, it has been proven \cite{bib10}\cite{bib11} that (elaborate) generic polynomials exist for a $P$-extension over $\mathbb{Q}$.

Hence the question still arises if, and if so which, "simpler" polynomials exist in the above sense, which offer much clearer insights into the structure of the needed extensions over $\mathbb{Q}$, in particular also for all intermediate fields. We shall prove by elementary means (and without computer algebra) that this is indeed possible, and furthermore, we have found an infinite family of solutions that even resides within the simplest possible type, namely irreducible \textit{pure} polynomials (irreducible \textit{binomials}),  of degree 8, in $\mathbb{Q}[X] $. 

Simplest infinite families of polynomials in this sense have been considered before for smaller groups, e.g. the simplest dihedral octic polynomials $X^8+(k^2+2)X^4+1$ with Galois group $D_8 = Dih_4$ of order 8, by Spearman and Williams in 2008 \cite{bib22}.
\newpage
In a broader context, irreducible binomials and radical extension fields have been investigated in the 1990s by several authors, including  E. Jacobson and W. Velez \cite{bib24}.
For irreducible binomials of degree $n$ over $\mathbb{Q}$, they characterized their Galois groups as so-called \textit{full} subgroups of a certain group 
$Hol(C_n)$ of order  $n \phi(n)$,
where the holomorph $Hol(C_n)$ is defined as the semidirect product $C_n \rtimes Aut(C_n)$, and a subgroup $G$ of this group is called \textit{full} if the projections of $G$ on the first and second factor are both surjective (but not necessarily homomorphisms). This amounts to a complete solution of this inverse Galois problem for radical extensions of the rationals, in a certain sense: if we determined exactly which of the subgroups of $Hol(C_n)$ are full, then we know the \textit{set} of Galois groups appearing for a fixed $n$.

However, the explicit isomorphism types of the occurring Galois groups for irreducible \textit{pure} polynomials of a given degree $n$, $f_c = X^n+c$, as well as their exact dependence on $c$, are intricate to extract from these results. 
On the contrary, we are aiming here at a direct characterisation of the family of those pure polynomials of degree 8 which give rise to the Pauli group as their Galois group over $\mathbb{Q}$.
We were able to check the compatibility of our work with these earlier results, and make some comments on the connection to the broader question which Galois groups can arise for \textit{pure} polynomials in chapter 4.

In chapters 2 and 4, we prove the main propositions that the irreducible polynomials $f_k = X^8 + k^2 \in \mathbb{Q}[X]$ have Galois group $P$ over $\mathbb{Q}$ \textit{if and only if} $k$ is not a square 
in $\mathbb{Q}$.
At the same time, for infinitely many values of $k$ the splitting fields of $f_k$ are distinct.
We also prove the fact that $Q_8$ cannot be realized by any  pure octic polynomial over $\mathbb{Q}$, but that it can be over a quadratic extension of $\mathbb{Q}$, as a natural realization of a $Q_8$- within a $P$-field extension over $\mathbb{Q}$. 
This fact has also been the foundation of earlier results by Lemmermeyer \cite{bib23}
on how to extract information about unramified quaternion extensions of quadratic number fields over $\mathbb{Q}$.
He also detailed the fixed fields of the 15 normal subgroups (out of 21 proper subgroups) of $P$ in terms of the discriminants of the fixed fields of its three $D_8$ subgroups, proving that these must be relatively prime.

As an intermezzo, chapter 3 connects these results to previous work into the embedding of $V_4$- into $Q_8$-extensions, leading back to Ernst Witt's classical paper in 1936. Chapter 4 also connects to the well-known general theorem of Schinzel on when a pure binomial has an \textit{abelian} Galois group, as well as to the aforementioned results about radical extensions of $\mathbb{Q}$ in general. Finally, in chapter 5 we link our results to those about embeddings of a triquadratic extension into a  $P$-extension, showing how to avoid some pitfalls that are due to the intricate nature of the connection to quaternion algebras and ternary quadratic forms.

\section{Construction of $P$ as a Galois group}\label{sec2}

Here we show explicitly how to construct an infinite family of normal $\mathbb{Q}$-extensions such that for each of them their corresponding group of automorphisms is isomorphic to the Pauli group. 
The Pauli group $P$ derives its name from its extensive use in physics, a 2-dimensional representation being generated by the three so-called Pauli matrices $X,Y,Z \in M_2(\mathbb{C})$, 
\begin{equation}
X = \begin{bmatrix}0 & 1 \\ 1 & 0\end{bmatrix}, 
Y = \begin{bmatrix}0 & -i \\ i & 0\end{bmatrix}, 
Z = \begin{bmatrix}1 & 0 \\ 0 & -1\end{bmatrix} \nonumber
\end{equation}   
involutory and Hermitian (hence unitary).
We have $iX=YZ=-ZY$, $iY=ZX=-XZ$, $iZ=XY=-YX$ (with $-E = [X,Y]= XYX^{-1}Y^{-1}$ and $iE=XYZ$), such that 
$N := \{\pm E, \pm iX, \pm iY, \pm iZ\}$
is a subgroup of the Pauli group $P$
isomorphic to the quaternion group $Q_8$.
Together with the coset 
$XN = \{\pm X, \pm iE, \pm Z, \pm Y\}$,  evidently it makes up all of $P$, which therefore has order 16 and appears as a semidirect product of the normal (index 2) subgroup $N$ with the cyclic group $H:=\langle X\rangle$ of order 2 (the action of $H$ on $N$ is given by $(iX)^X = iX, (iY)^X = XiYX= -iY = (iY)^{-1}$). 
We also note: the elements of $P$ have orders 1,2 or 4 only,
and the center $Z(P)$ of $P$ is given by $Z(P)=\langle iE\rangle$, and
therefore of type $C_4$. 
For further algebraic presentations of $P$, as well as a remark on other usages of the name Pauli group, different from the classical one we use, we refer the reader to the appendix.\newline

We begin by recalling the fact that a pure polynomial of the form $X^n+k^2$ with $k^2 \in \mathbb{Q}^*$ is irreducible over $\mathbb{Q}$ if and only if two conditions are fulfilled \cite{bib6}:\newline

(a) there is no prime factor $q$ of $n$ such that $-k^2$ is a $q$-th power in $\mathbb{Q}$

(b) if $n$ is divisible by 4, there is no $\lambda \in \mathbb{Q}$ such that $-k^2 = -4\lambda^4$\newline

For $n=8$ ie. $q=2$, (a) is trivially respected over $\mathbb{Q}$, and (b) gives the condition $k \neq \pm 2\lambda^2$.
We may now further suppose $k>0$ and we also have to require $k$ not to be a square in $\mathbb{Q}$ (since it can easily be seen that otherwise our construction would definitely fail, leading to a Galois group of order 8).
These constraints will give us the sufficient condition on $k$ so that the $\mathbb{Q}$-extensions for the family of polynomials $f_k := X^8+k^2$ yield the desired Galois group isomorphic to $P$: 

\begin{equation}
 k \neq \lambda^2,2\lambda^2;  
0 < k, \lambda \in \mathbb{Q}^*
\end{equation}
\newline
We note that (1) is at the same time exactly the condition for a pure quartic polynomial $X^4+k^2$ to be irreducible, to have $V_4$ as Galois group \textit{and} to possess a splitting field $F$
\begin{equation}
    F = \mathbb{Q}(w\sqrt{k}), w^2 = i
\end{equation}
which is \textit{dependent}  on $k$, 
like also the $k$-dependent quadratic intermediate fields 
$\mathbb{Q}(\sqrt{k/2})=\mathbb{Q}(\sqrt{2k})$ and $\mathbb{Q}(i\sqrt{2k})$, alongside $\mathbb{Q}(i)$. For all $k$ that are squares in $\mathbb{Q}$, the splitting field $F$ is the $k$-independent $\mathbb{Q}(w)$. Hence (1) is related to the fact that we need seven quadratically independent quadratic extensions of $\mathbb{Q}$ to produce seven different biquadratic $V_4$-extensions, only thus allowing us to build a triquadratic $E_8 = C_2^3$-extension, which is the maximal abelian subextenion of the Pauli extension: $P/C_2 \cong E_8$, for the polynomials $f_k$, as detailed below.
\newpage

We shall use the following notation, reserving $a$ for
the "first" root of $f_k =X^8+k^2$ : 
\begin{equation}
a := zv
\end{equation}
with:
\begin{equation}
v^2 := +\sqrt{k}  \notin \mathbb{Q} , v^4 =k \in \mathbb{Q}, v > 0 \nonumber
\end{equation}
\begin{equation}
z  := \exp{(2 \pi i/16)} = (R_+ + iR_-)/2 = \zeta _{16} \nonumber
\end{equation}
\begin{equation}
w := \exp(2 \pi i/8) = z^2  
    = (1+i)r/2 =  \zeta _8\nonumber
\end{equation}
\begin{equation}
r  := +\sqrt{2} = R_+R_- = w+\Bar{w}, rv^2 = \sqrt{2k} \nonumber
\end{equation}
\begin{equation}
R_\pm := +\sqrt{(2 \pm r)}, R_+ = z+\Bar{z}, vR_+ = a+\Bar{a},  \nonumber
\end{equation}
\begin{equation}
 viR_- = v(z-\Bar{z}) = a-\Bar{a} 
= i\sqrt{2\sqrt{k}-\sqrt{2k}} \nonumber
\end{equation}
\newline
For an octagon inscribed into the unit circle, $R_- = (z-\Bar{z})/i$ is the length of an edge,  and $r$ (resp. $R_+ $)  the length of a secant missing one  corner (resp. two corners). 

The condition on $v^2$ is equivalent to requiring $k$ not to be a square in $\mathbb{Q}$, ie. $k^2 \neq \lambda ^4$ for any $\lambda \in \mathbb{Q}$. For instance, all numbers  $k = p^{2\nu +1}$ for an odd prime $p \geq 3$ and integer $\nu \geq 0$ fulfil the conditions (1) 
For definiteness, the reader may anchor on the two simplest cases, $f_3 = X^8+9$ or $f_5 = X^8+25$ and 
for guidance, we refer to Fig. 1. To our knowledge this is the first complete depiction of the matching of three lattices: the subgroup structure of $P$, the 
$\mathbb{Q}$-extensions fixed by (automorphism groups isomorphic to) those subgroups, and the corresponding generators of these subgroups, drawn from the elements of $P$. It renders explicit, using the numbers $i,r,v^2$ from (3), the contents of table 1 in \cite{bib23}.

\begin{figure}[!ht]
\centering
\includegraphics[width=.99\linewidth]{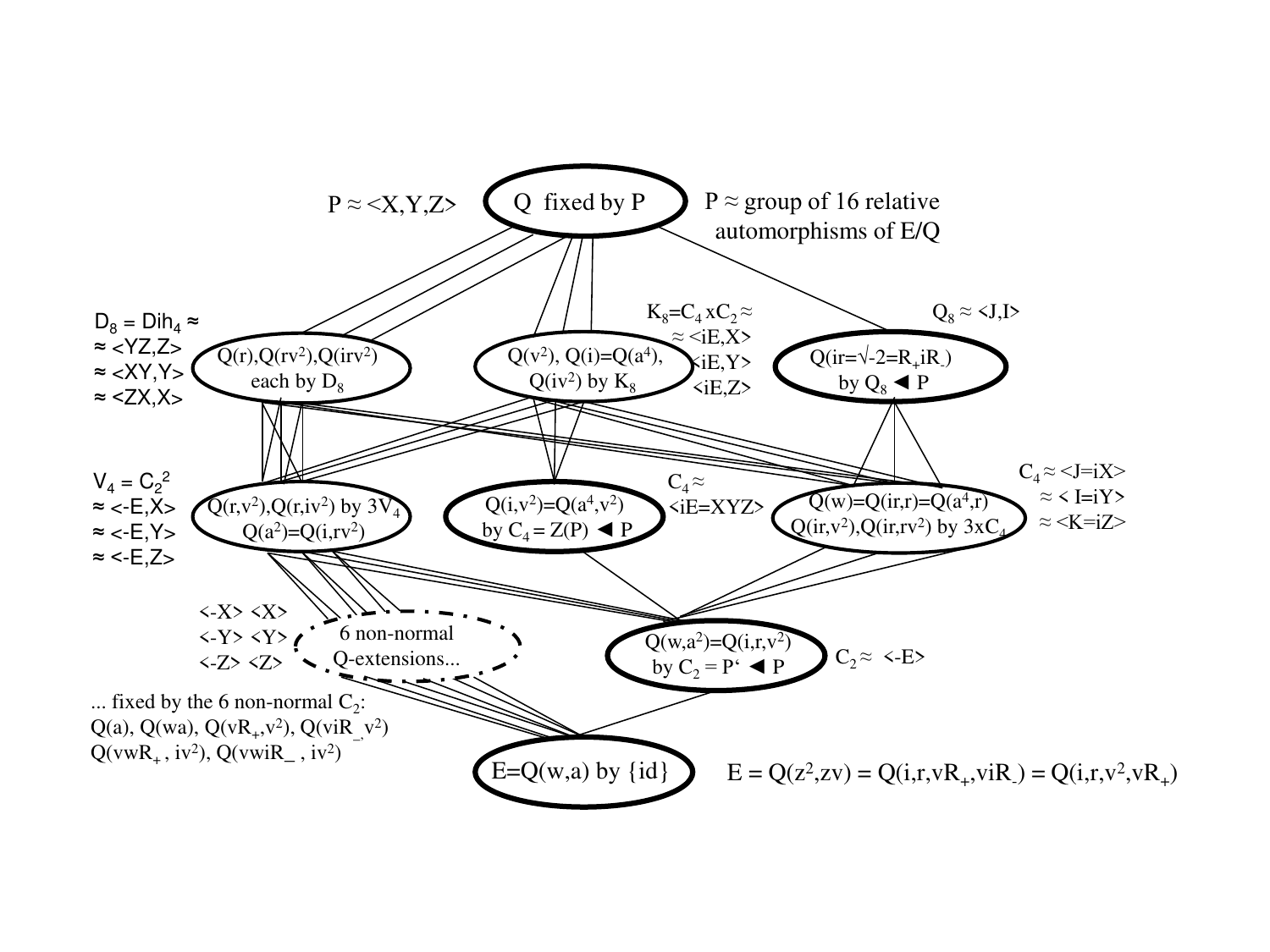}
\caption{$\mathbb{Q}$-extensions fixed by subgroups of the Pauli group, as well as their subgroup generators in terms of the Pauli matrices $X,Y,Z$ (relations understood). This is the completion of earlier fragments of $P$ as given e.g. in \cite{bib12}. We refer to (3) for $r=\sqrt{2},v^2=\sqrt{k}$ and other notations. }\label{fig1}
\end{figure}

To summarize, $k \neq 2\lambda^2$ arises due to requiring $f_k$ to be irreducible over $\mathbb{Q}$, and $k \neq \lambda^2 $  due to our aim to prove that the Galois group of $f_k$ over $\mathbb{Q}$ yields $P$, and no other group (which, incidentally, for the polynomial $X^8 + \lambda^4$ would have to be of order 8, since 
$\mathbb{Q}(a^2)  = \mathbb{Q}(w)$,
and furthermore abelian, according to Schinzel's result in section 4).
The roots of $f_k$ are $aw^m, m=0,...,7$.
Using these definitions, we now turn to our \newline

\begin{proposition}
The extension $E/\mathbb{Q}$ given by the splitting field $E=\mathbb{Q}(a,w)$ of any pure polynomial $f_k = X^8+k^2$ with $k$ as in (1) has the Pauli group $P$ as its Galois group.
Thus the subgroup lattice of the Pauli group $P$ mirrors (one-to-one, and order inverting) 
the lattice of intermediate fields of the extension $E/\mathbb{Q}$. 
\end{proposition}

\begin{proof}
Let $G$ be the Galois group $Gal(E/\mathbb{Q})$, where
the splitting field $E=\mathbb{Q}(a,w)$ of $f_k$ contains, for fixed $k$, the intermediate field
$L:= \mathbb{Q}(w,a^2) = \mathbb{Q}(i,r,v^2) = \mathbb{Q}(ir,r,v^2)$, with $a^2 = wv^2$ and $\mathbb{Q}(w)=\mathbb{Q}(i,r)$. 
Since $f_k$ is irreducible, 
$\mathbb{Q}(a)/\mathbb{Q}$ is of degree 8. 
$L/\mathbb{Q}$ is also of degree 8, a triquadratic Galois extension with well-known Galois group $E_8 = C_2^3$. 
In particular, for any element $t\in G$,  
$t^2$ acts trivially on $L$. Hence $t^2 \in Gal(E/L)$, which has order equal to
 the degree of $E/L$, ie. at most 2, so it follows $t^4=1$, and thus $G$ has no element of order 8.\newpage
 
After proving $E \neq L$, we can assert that:
\newline 
- $[E:L] = 2$ and $[E:\mathbb{Q}] = 16$, \newline
- $w \notin \mathbb{Q}(a)$,  \newline
- $\mathbb{Q}(a)/\mathbb{Q}$ is not a normal extension. \newline

But assuming that the splitting field was $E=L$, the equality of the two degree 8 extension fields $\mathbb{Q}(a) = L = \mathbb{Q}(w,a^2)$ would follow. Thus, like $L/\mathbb{Q}(w)$, the extension $\mathbb{Q}(a)/\mathbb{Q}(w)$ would be a quadratic extension. Hence we would find that $a^2 \in Q(w)$, the field of the 8th roots of unity, and thus $\mathbb{Q}(w)=\mathbb{Q}(w,a^2=w\sqrt{k})$, a contradiction. \newline


Out of the 14 groups of order 16, the Pauli group $P=Q_8 \rtimes _2 C_2$, among the 9 non-abelian groups \cite{bib5}, is the only one with exactly one subgroup $Q_8$.

The only three other groups of order 16 with one or more $Q_8$ as  subgroups are: \newline
- $Q_{16}$, the generalised quaternion group of order 16, \newline
- $SD_{16} = QD_{16} = Q_8 \rtimes C_2$, the semi- or quasi-dihedral group of order 16, \newline 
- $Q_8 \times C_2$.\newline 
Among these, the first two have a cyclic subgroup $C_8$ (and lack a factor group $E_8$), in contrast to $P$ , whereas the last 
has normal subgroups only (as opposed to $P$).\newline

Thus to show a group $G$ of order 16 to be isomorphic to the Pauli group, it suffices to demonstrate that simultaneously:\newline
a) $G$ has no element of order 8 \newline 
b) $G$ has non-normal subgroups\newline
c) $G$ possesses a subgroup of type $Q_8$. \newline

We address these conditions in turn, showing that the Galois group 
$G = Gal(E/\mathbb{Q})$ with $E$ as the splitting field of  $f_k$ over $\mathbb{Q}$, satisfies them for any k respecting (1):\newline

ad a) This has already been ascertained at the beginning of the proof.
\newline

ad b) $Gal(E/\mathbb{Q}(a))$ is a non-normal subgroup of $G=Gal(E/\mathbb{Q})$, since we proved above that the extension
$\mathbb{Q}(a)/\mathbb{Q}$ is not normal. The same holds for
$\mathbb{Q}(wa)/\mathbb{Q}$. \newline

 
ad c) We consider the fixgroup $H$ of the subfield $\mathbb{Q}(ir)$, so $H = G(E/\mathbb{Q}(ir))$. Since $\mathbb{Q}(ir)/\mathbb{Q}$ is of degree 2, $H$ is a subgroup of order 8. The extension 
$E/\mathbb{Q}(ir)$ contains the intermediate fields 
$K_{j=1,2,3} =\mathbb{Q}(ir,r), \mathbb{Q}(ir,v^2), \mathbb{Q}(ir,rv^2)$ which are easily shown to be distinct. 
\newline
We start by observing that none of the fields $K_j$ contains $a^2=wv^2$.
It suffices to show that the groups $G(E/K_j)$ of order 4 are all cyclic, because the quaternion group $Q_8$ is the only group of order 8 containing three different normal subgroups of type $C_4$.

Now, assuming to the contrary that at least for one $j$, ie. for one of the groups $G(E/K_j)$ we have: 
for all $s  \in G(E/K_j): s^2 = 1$.

Then for each  $s \in G(E/K_j)$, we would have  $sa = ca$ with an 8th root of unity  $c = c(s)$. 
Hence $s^2 =1$  gives $cs(c) =1$, ie. $s(c) = c^{-1}$.

For every $j$, $s  \in G(E/K_j)$  fixes the element $ir$, so that on the 8th roots of unity, either $s$ acts trivially, or by exponeniating by 3.

In particular, it follows that $c^4 = 1$ for the $c = c(s)$ from above. Because of $s(a^2) = (ca)^2 = c^2 a^2$ we find, honoring the initial remark, an $s$ with
\begin{equation}
       c(s)^2 = -1.
\end{equation}

Thus for this $s$ we get:
$sa  =  \pm ia$. Hence:
\newline

Case $j=1$:   here $ w \in K_1$ and we have   
$c  = s(c) = c^{-1}$, a contradiction to (4).
\newline

Case $j=2$:  here we have \newline
$s(a^2) = s(wv^2) = s(w) s(v^2) = s(w) v^2 = w^3 v^2 = w^2 wv^2 = w^2 a^2 = i a^2$,  \newline
a contradiction to $s(a^2) = c(s)^2
a^2$  and (4). 
(The alternative $s(w) = w$ is excluded because of (4))\newline

Case $j=3$: finally, here we observe 
$s(a^2) = s(wv^2) = s(w) s(v^2) = w^3 s(v^2) = w^3
s(r^{-1}) s(rv^2) = w^3 s(r)^{-1} rv^2 = w^3 (-r)^{-1} r v^2 =  -w^3 v^2 = -w^2 wv^2 = -i a^2$, again a contradiction to (4) and to the line above it.
\newline

To summarise, this proves that $H=Q_8$, as required.
The case for $K_1=\mathbb{Q}(w)$ can independently be derived from Kummer theory and the existence of the primitive 8th root of unity $w \in K_1$, requiring $E/K_1$ to be cyclic of degree 4 (since $a^4 \in \mathbb{Q}(w)$, but  $a^2$ is not). 
\end{proof}

As to other intermediate fields of the extension 
$E/\mathbb{Q} =\mathbb{Q}(w,a)/\mathbb{Q}$, fixed by groups isomorphic to $C_4$, we can assert the following: the extension 
$E/\mathbb{Q}(i) =\mathbb{Q}(a,v^2)/\mathbb{Q}(i)$
results from the Galois extension
$\mathbb{Q}(i,v^2)/\mathbb{Q}(i) $
by composition with 
$\mathbb{Q}(a,i)=\mathbb{Q}(a)$.
Since $\mathbb{Q}(i,v^2)/\mathbb{Q}(i) $ is of degree 2,
the intersection of $\mathbb{Q}(i,v^2) $ with $\mathbb{Q}(a) $
can only be $\mathbb{Q}(i) $.
According to the translation theorem of Galois theory, we have thus:
\begin{equation}
 Gal(E/\mathbb{Q}(i)) = 
 Gal(\mathbb{Q}(a,v^2)/\mathbb{Q}(i,v^2)) \times
 Gal(\mathbb{Q}(a,v^2)/\mathbb{Q}(a))
\end{equation}

As $a^2=wv^2$, $a^2$ cannot be an element of $\mathbb{Q}(i,v^2)$, hence the first factor  (as group of order 4) is isomorphic to $C_4$, and the second of type $C_2$.
(In passing, it also follows that 
$\mathbb{Q}(a)/\mathbb{Q}(i)$
is of type $C_4$).

We had shown that the subgroup 
$Gal(E/\mathbb{Q}(ir)) $ is of type $Q_8$.
This makes up for three (out of the four) subgroups of type $C_4$  in  $P$, none of these three equal  to the center $Z(P)$. Therefore the subgroup  
$Gal(\mathbb{Q}(a,v^2)/\mathbb{Q}(i,v^2)$,
 proven above of type $C_4$, and certainly not contained in 
 $Gal(E/\mathbb{Q}(ir))\cong Q_8$, must be the center $Z(P)$:

\begin{equation}
 Gal(E/\mathbb{Q}(i,v^2)) = 
 Gal(\mathbb{Q}(a,v^2)/\mathbb{Q}(i,v^2) \cong Z(P) \cong C_4
\end{equation}
\newline
For the other subgroups of $P$ and their corresponding fixfields, we refer the reader to Fig.2, where the 7 quadratic and 7 biquadratic intermediate fields of $L=\mathbb{Q}(ir,r,v^2)=\mathbb{Q}(w, w\sqrt{k})$ are analogous to those of any standard triquadratic $\mathbb{Q}$-extension, producing a Galois group $Gal(L/\mathbb{Q}) \cong E_8$. These 14 subfields correspond to the number of 2- (resp. 1-) dimensional subspaces of the 3-dimensional $C_2-$vector space $C_2^3$.

$L/\mathbb{Q}$ is the only normal extension of degree 8.
The other 4 of 6 (non-normal) degree-8-extensions can be inferred in a similar way as we inferred $\mathbb{Q}(a)$ and $\mathbb{Q}(wa)$, and are explicitly given in Fig.2.\newline
Thus we have arrived at a constructive proof, that for any $k$ fulfilling (1), all members of the infinite family of irreducible non-abelian binomials $f_k = X^8+k^2$ are the simplest (ie. pure) polynomials that lead to the Pauli-group as their Galois group. The context of $P$ appearing for $n=8$ is further explored in section 3 and 4.

\begin{figure}[!ht]
\centering
\includegraphics[width=1.1\linewidth]{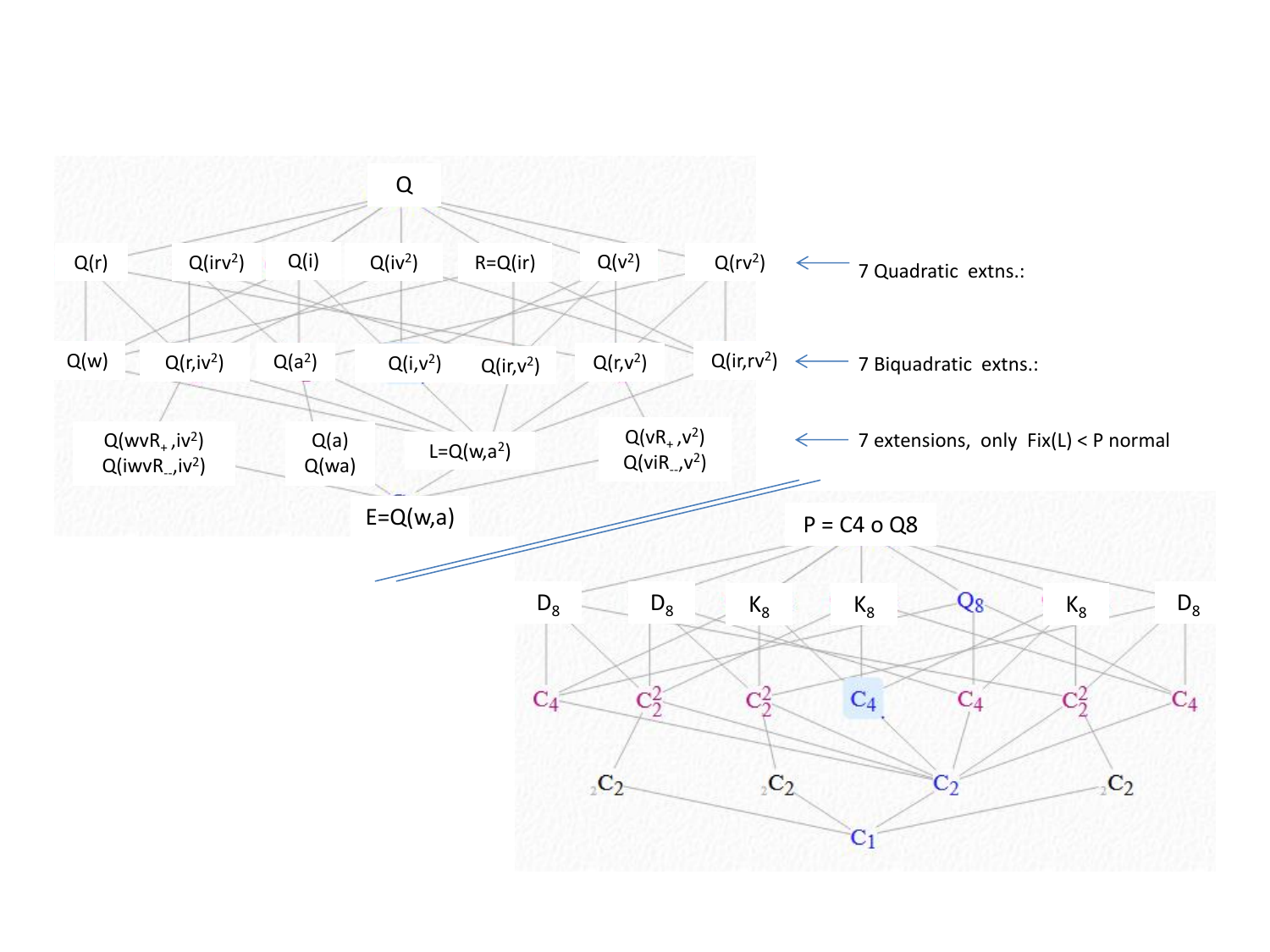}
\caption{Matching of intermediate fields to subgroups of $P$. Note the breaking of the horizontal symmetry, due to $P/C_2 \cong E_8=C_2^3$, among the subgroups of $P$ of order 4 and 8, when this figure is compared to Fig.1. We refer to (3) for $r=\sqrt{2}, v^2=\sqrt{k}, w^2=i$ and other notations.}\label{fig2}
\end{figure}
As we have seen, the single $C_2$ subgroup (out of 7 such) of $P$ that is normal, its commutator $P'$, gives rise to the factor group $P/P' = E_8 = C_2^3$, elementary, ie. of exponent 2, and corresponds to the intermediate field $L = \mathbb{Q}(w,a^2) = \mathbb{Q}(w,w\sqrt{k})$.
The last and final quadratic extension $E/L=\mathbb{Q}(w,a) / L$  over $L$ thus can be interpreted as 
breaking the complete "horizontal" symmetry between the seven (bi)quadratic extension fields generated by 
one (or two) elements out of
$\{i,r,v^2\} = \{i,\sqrt{2},\sqrt{k}\}$, a symmetry that is evidenced in the well-known symmetric subgroup lattice pertaining to the elementary group $E_8$, as visible in Fig 2, as opposed to Fig. 1 which is the more lucid one for revealing the structure of $P$. 


\section{Relation to $Q_8$ extensions}\label{sec3}
  
Since the occurence of a Galois subgroup
$Q_8 \cong Gal(E/\mathbb{Q}(ir))$ played such a central role in our proof,
as well as for the general context of constructing quaternion group extensions over $\mathbb{Q}$, and over extensions of $\mathbb{Q}$, it is intriguing to question what extensions $E/R$ over a ground field $R$ of characteristic 0 (or at least $\neq 2$), of a given biquadratic intermediate field extension $L/R$ (with Galois group isomorphic to $V_4$, in our case $R=\mathbb{Q}(ir)$), give rise to a  Galois group $Q_8$. As is well known, in 1936 E. Witt solved this problem by methods using quadratic forms and crossed products \cite{bib7}. Specialising his theorem to our case, we have: \newline

\begin{proposition}[Witt 1936]
Let $L=R(\sqrt{a_1},\sqrt{a_2},\sqrt{a_3}=\sqrt{(a_1a_2)^{-1}})$ be a biquadratic extension over a ground field $R$ with $Gal(L/R)=V_4$. Then 
$L/R$ can be embedded into a quaternion extension $F/R$ 
if and only if the quadratic forms 
$[a_1,a_2,a_3 ] \cong  [1,1,1]$ are equivalent over the ground field $R$. \newline
Furthermore, if $T = (t_{ij})$ is a matrix with  $det(T)=1$
which transfers $[a_1,a_2,a_3 ]$ into $[1,1,1]$, then all the quaternion extensions $F/R$ containing $L$ are obtained by $F=R(\sqrt{\rho \beta})$ with
$\beta = 1+t_{11}\sqrt{a_1}+t_{22}\sqrt{a_2}+t_{33}\sqrt{a_3}$ and a certain non-zero element $\rho$ of the ground field $R$ (and vice versa). \newline
\end{proposition} 
The Galois $P$-extensions of $\mathbb{Q}$ that we constructed in section 2, for any $k$ respecting (1), 
lead to a family of $Q_8$-extensions of 
$\mathbb{Q}(ir) = \mathbb{Q}(\sqrt{-2})$.
Here we explicate in detail that our construction is fully congruent with the result of Witt in the case of the ground field being $R=\mathbb{Q}(ir)$ and $a_1=2, a_2=k, a_3=1/2k$.\newline
To begin with, indeed $[2,k,(2k)^{-1}]\cong[2,k,2k]\cong [1,1,1]$ are $\mathbb{Q}(ir)$-equivalent forms, since by the Witt relations in $\mathbb{Q}$: $[1,1,1] \cong [2,2,1]$, and because, $-2$ being a  square in $\mathbb{Q}(ir)$, we obtain 
$ [2,2,1] \cong [-1,-1,1]\cong [-1,1,-1] \cong [-1,k,-k] \cong [2,k,2k]$.


A matrix $T$ over $\mathbb{Q}(ir)$ that establishes the equivalence, $T^t Diag(2,k,1/2k)T = E$, 
of the two quadratic forms $[2,k,2k] \cong [1,1,1]$, normed to $det(T) = 1$, abbreviating 
$K=k+\frac{1}{2}, 
\kappa=k-\frac{1}{2}$, is given by :
\begin{equation}
T = -(1/2)\begin{bmatrix}1 & 1 & 0\\ -K/k & K/k & -\kappa ir/k \\ \kappa ir & -\kappa ir& -2K \end{bmatrix}, 
 \nonumber
\end{equation}  
Its diagonal elements determine the $\beta$ of Witt's theorem. We get
\begin{equation}
\beta = 1 - \frac{1}{2}\sqrt{2}-\frac{K}{2k}\sqrt{k} + \frac{K}{2k}\sqrt{2k} \nonumber
\end{equation} 
\begin{equation}
= (2-\sqrt{2})(k + K\sqrt{2k}/2)/2k 
\nonumber
\end{equation}
\begin{equation}
= (R_-^2 \sqrt{k}) (4\sqrt{k} + (2k+1)\sqrt{2})/8k 
\nonumber
\end{equation}
\begin{equation}
= - (viR_-)^2 (w(1+\sqrt{2k}))^2/4kir \newline
\end{equation}   
with $w=(1+i)r/2$, c.f. (3). Choosing the non-zero element $\rho = -4kir$ of $\mathbb{Q}(ir)$ we obtain
\begin{equation}
\rho \beta = (viR_-)^2 (w(1+\sqrt{2k}))^2.\nonumber
\end{equation}
Since the second factor on the right side is a square of an element of $L$, it follows:
\begin{equation}
L(\sqrt{\rho \beta}) = L(viR_-).  
\end{equation}
Referring to (3) once again we observe $a - \Bar{a} = viR_-$. 
This $viR_-$ being non-zero, and $w$ contained in $L$, it is easily seen that 
\begin{equation}
L(viR_-) = L(a-\Bar{a})= \mathbb{Q}(w,a) = E. \nonumber 
\end{equation}
Thus by (8) our splitting field $E$ has indeed the desired form, with $\rho$ and $\beta$ explicitly given as above:
\begin{equation}
E = L(\sqrt{\rho \beta}) =  \mathbb{Q}(ir)(\sqrt{\rho \beta}),  
\end{equation}
We have thus automatically found pure polynomials $f_k$ that represent $Q_8$-extensions 
over $\mathbb{Q}(ir)$, while no pure polynomial of degree 8 can be found to give rise to a Galois $Q_8$ extension over $\mathbb{Q}$, as is proven in the next section.


\section{Galois groups of pure polynomials}\label{sec4}

Which finite groups can be realized as Galois groups by \textit{pure} polynomials over $\mathbb{Q}$? 
In the context of this rather general question, this section serves a double purpose: on one hand, 
we will show that the construction in section 2 is not only sufficient, but indeed necessary for $P$ occurring as a Galois group of a pure polynomial of degree 8 over $\mathbb{Q}$.
On the other hand, we shall prove that the subgroup $Q_8 \triangleleft P$ cannot be realized over $\mathbb{Q}$ by a pure polynomial of degree 8. 
This situation means that our Pauli-construction via $f_k = X^8 + k^2$ proves that $Q_8$ is exactly one quadratic $\mathbb{Q}$-extension away from being realizable itself by a pure polynomial over $\mathbb{Q}$. Thus we proceed with \newline

\begin{proposition}
Let $f = X^8 + c$ be a polynomial over $\mathbb{Q}$ with $c\ne 0$. Then we have:\newline

a. If $f$ is reducible then its Galois group can be neither $P$ nor $Q_8$.




b. For $c$ not a square in $\mathbb{Q}$, the Galois group of $f$ over $\mathbb{Q}$ is not $P$.\newline

\end{proposition}

Propositions 1 and 3 taken together now give us our primary result (recalling our statement before prop.1 that for any $k=m^2$, the Galois group of $f_k$ has order 8 and thus cannot be $P$): \newline
\begin{proposition}
The order 16 Pauli group $P$ occurs as the Galois group of a pure polynomial $f\in \mathbb{Q}[X]$ of degree 8  if and only if $f=f_k=X^8+k^2$, with $k$ satisfying (1), ie. $k\neq \lambda^2, 2\lambda ^2;  0 < k,\lambda \in \mathbb{Q}$.
\end{proposition}
\begin{proof} 
(of prop.3): 

ad a.: We show the equivalent statement: If the Galois group of the polynomial $f$ is $P$ or $Q_8$, $f$ must be irreducible. We assume to the contrary that $f$ be reducible, and that its Galois group $G$ is of type $Q_8$ or $P$: 
\newline
Then the irreducible factors of $f$ can only have degree 1,2, or 4. 
Not all factors can be of degree $\le 2$, else $G$ would be abelian.
Thus we have $f=gh$ with an irreducible $g$ of degree 4 (and $h$ of degree 4, possibly reducible). The splitting field $E$ of $f$ is the composite of the splitting fields $E_g$ and $E_h$ of $g$ and $h$, respectively. Here $E_g=E$ is impossible, else $G$ would be the Galois group of an irreducible polynomial of degree 4, and thus isomorphic to a subgroup of $S_4$, 
which excludes $G = Q_8$, but also $G = P$, since $P$ contains $Q_8$ as a subgroup. 
So in case $G = Q_8$ it follows $E_g:\mathbb{Q}=4$, and in case $G = P$ it follows $E_g:\mathbb{Q}=4$ or $E_g:\mathbb{Q}=8$. If  $E_g:\mathbb{Q}=4$, it is abelian. But the same holds also if $E_g:\mathbb{Q}=8$, since the only factor group of $P$ having order 8 is (elementary) abelian. If the factor $h$ in $f=gh$ is irreducible, then applying the above reasoning to $h$ instead of to $g$ shows that $E_h/\mathbb{Q}$ is abelian. But this evidently is also the case if $h$ is reducible.
With both $E_g/\mathbb{Q}$ and $E_h/\mathbb{Q}$ abelian, the compositum $E/\mathbb{Q}$ is abelian, too, which contradicts the assumption $G=G(E/\mathbb{Q}) = P$ or $=Q_8$.
\newline



ad b.: We write $c=tk^2$, with square free $t\ne 1; t\in \mathbb{Q}^*; 0<k\in \mathbb{Q}$, and let $a$ be a root of $f=X^8+c$.
Then 
$(a^4)^2 = -c =  -tk^2 = (i\sqrt{t}k)^2$, or
\begin{equation}
a^4 = \pm i\sqrt{t}k
\end{equation}
We now assume that the Galois group of $f$ was of order 16.
Then for degree reasons, the splitting field $E=\mathbb{Q}(w,a)$ has degree 4 over $\mathbb{Q}(w)$, hence $a^4\in \mathbb{Q}(w)$, and from (10) follows
\begin{equation}
\sqrt{t} \in \mathbb{Q}(w).
\end{equation}
So  $\mathbb{Q}(\sqrt{t})$ must be one of the three quadratic subfields of $\mathbb{Q}(w)$, and this implies: 
$t=2x^2$ or $t= -x^2$ or  $t=-2x^2$ for some $x\in \mathbb{Q}$. For square free $t$ it follows  $t= 2, -1$, or $-2$. But $t=-1$ ie. $c=-k^2$ is taken care of by the reducible case 3a, and thus
\begin{equation}
c = -2k^2, \text{resp.   }  c = +2k^2 .
\end{equation}

Furthermore, due to $(a^2)^4 = a^8 = -c = 2k^2$, resp. $-2k^2$ we have\newline
$(a^2)^4 = (\sqrt{k}\sqrt[4]{2})^4, \text{  resp.  } 
(a^2)^4 = (w\sqrt{k}\sqrt[4]{2})^4$, therefore in both cases:
\begin{equation}
\sqrt{k}\sqrt[4]{2} = \sqrt[4]{2k^2} \in E = \mathbb{Q}(w,a).
\end{equation}
As can be readily seen, the field $F:=\mathbb{Q}(\sqrt[4]{2k^2})$ has degree 4 over $\mathbb{Q}$.
The subgroup $G(E/F)$ of $G=G(E/\mathbb{Q})$ thus has order 4 and is not normal in $G$ (as $F/\mathbb{Q}$ is not normal). However, in the Pauli group $P$ every subgroup of order 4 is normal, thus $G\ne P$, q.e.d. 
\end{proof}
We note from (12) in the proof that Galois groups of order 16 for irreducible $f$ occur only for the cases $c= 2k^2, -2k^2$, if $c$ is not a square in $\mathbb{Q}$. It can be shown that they turn out to be the groups $Dih_8 = D_{16}$ and $QD_{16}$, respectively.\newline
Now we briefly turn to the question of realizability of $Q_8$ and $K_8= C_4 \times C_2$ as Galois groups for pure polynomials of degree 8: \newline

\begin{proposition}
If a group $G$ of order 8 is realizable as Galois group of an irreducible, pure polynomial $f=X^8+c$ over $\mathbb{Q}$, then $G$ is abelian and of isomorphism type $K_8 = C_4 \times C_2$.
In particular (observe also prop. 3a), the quaternion group $Q_8$ is not realizable as a Galois group of any pure polynomial of degree 8 over $\mathbb{Q}$. 
\end{proposition}

\begin{proof}
From the assumptions it immediately follows that for every zero $a$ of $f$, the extension 
$\mathbb{Q}(a)/\mathbb{Q}$ is normal and contains the field $\mathbb{Q}(w)$ of the 8th roots of unity. Thus $a$ cannot be real, and we have $c>0$.

Since $a^8 \in \mathbb{Q}(w)$ and $\mathbb{Q}(a)/\mathbb{Q}(w)$ has degree 2, $a^2$ must lie in $\mathbb{Q}(w)$, and in fact 
$\mathbb{Q}(a^2)=\mathbb{Q}(w)$. Now if $v$ denotes the real 4th root of $c$, we have $(a^2)^4=a^8= -c=(wv)^4$, and the element 
$a^2 \in \mathbb{Q}(w)$ must be $v$, up to an 8th root of unity, from which it follows that $v \in \mathbb{Q}(w)$.
Since $\mathbb{Q}(v)$ is real, 
it is contained in the maximal real subfield of 
$\mathbb{Q}(w)$, which is $\mathbb{Q}(\sqrt{2})$.
 In particular, $X^4-c$ cannot be irreducible, therefore $c$ must be a square in $\mathbb{Q}$, $c=d^2$ with some $d>0$ in $\mathbb{Q}$.
It follows $d=v^2$. 
If  $d$ is not a square in $\mathbb{Q}$, we get
$\mathbb{Q}(\sqrt{d})=\mathbb{Q}(v)=\mathbb{Q}(\sqrt{2})$, hence $d=2x^2$ with an $x \in \mathbb{Q}$. But then $c=d^2=4x^4$, a contradiction to the irreducibility of $X^8+c$.
Thus we are left with the conclusion that 
$v$ lies in  $\mathbb{Q}$. But then $c^2=(v^4)^2$ is an 8th power in $\mathbb{Q}$, and the Galois group of $f$ must be abelian, according to Schinzel's theorem (proposition 6) below, which also states that this group must be of isomorphism type  $K_8=C_4 \times C_2$.
\end{proof}

It turns out that $Q_8$ is the Galois group of an even octic polynomial over $\mathbb{Q}$, e.g. \newline
$x^8-12x^6+36x^4-36x^2+9$ \cite{bib5}, however, no simpler polynomial seems to be able to produce $Q_8$.
For the special case of a square constant term it has been shown that no polynomial of the form $X^8+bX^4+d^2$ has $Q_8$ as Galois group over $\mathbb{Q}$ \cite{bib21}.
\newline

By these considerations, one is naturally led to consider the more general question: 
What characterises Galois groups of irreducible pure polynomials $X^n-\gamma, \gamma \in \mathbb{Z}$  of a given order $n$ over $\mathbb{Q}$ in general (in particular for the cases $4|n$), apart from being solvable groups? Which of the finite groups do (or do not) occur for any $\gamma$? \newline

The \textit{abelian} case has been resolved in a beautiful theorem by Schinzel in 1977 \cite{bib8}, building on the work of Darbi in 1925 \cite{bib9}. For the convenience of the reader we include here only the special case of the ground field $\mathbb{Q}$, for which his theorem  on "abelian binomials" states:\newline

\begin{proposition}[Schinzel 1977]
Given a pure polynomial $f = X^n-\gamma, \gamma \in \mathbb{Q}$, let $G$ be its Galois group over $\mathbb{Q}$. Then $G$ is abelian if and only if $\gamma ^2$ is an $n$-th power in $\mathbb{Q}$ (or equivalently in $\mathbb{N}$).

In this case, and provided $f$ is irreducible, $G$ is cyclic, if 4 does not divide $n$, and otherwise it is isomophic to the direct product $K_n = C_2 \times C_{n/2}$. \newline
\end{proposition} 

Remark: Above we used only a small part of Schinzel's result, namely if $\gamma ^2$ is an $n$-th power in $\mathbb{Q}$, then the Galois group of $f$ is abelian.  
This is in fact easily proven:  If  $\gamma^2 = d^n$ with $d \in \mathbb{Q}$, and if $a$ denotes a zero of $f$, then  $a^{2n}  =  d^n  =  \sqrt{d}^{2n}$, and hence  $a = z \sqrt{d}$ with  $z^{2n} = 1$. Therefore $a$ is an element of the extension field $\mathbb{Q}(z, \sqrt{d})$, which is abelian over $\mathbb{Q}$.

In our context with irreducible $f=X^n-\gamma$, it can be readily seen that over $\mathbb{Q}$ these "rare" abelian cases can arise only for $n=2^s, s\geq 1$. 
The cyclic case occurs for $n=2$ only, and for $s\geq 2$ these groups are all of the form $K_n = C_2 \times C_{n/2}$, starting with $K_4=V_4$, the Klein group of order 4, for $n=4$.\newline

Considering the Galois groups of irreducible polynomials $f=X^n-\gamma = X^n +c$ for $n=2^s$, $m = n/2$, corresponding to $s, s-1$, \textit{for a fixed}  $c$, we thus realize that  those groups for  $c=d^{n/4}=d^{m/2} \neq \mu ^m$ that were abelian for $X^m+c$ have become non-abelian for $X^n+c$, whereas those with $c=d^{n/2}=d^m$ remain abelian,  after having moved up $n$ by one power of $2$, ie. from $2^{s-1}$ to $2^s$. 

For our case $n=8$, 
in this sense the Pauli group for irreducible $X^8+c$
"lies above" $V_4$ as Galois group for $X^4+c$ for $c=k^2 \neq \mu ^4$, whereas the abelian $K_8$ "lies above" $V_4$ for $c=\mu ^4$, ie. a forth power in $\mathbb{Q}$. 

For the next level $n=16$, analogously,  nonabelian groups (including the modular group of order 16, $M_4(2)$, and $C_8 \circ P$) lie above 
$K_8$ (as Galois group for $X^8+c$) for $c=d^4 \neq \mu ^8$, whereas the abelian $K_{16}$ lies above $K_8$ for $c=\mu ^8$, for some $\mu \in \mathbb{Q}$.

In line with the results of Jacobson and Velez cited in the introduction \cite{bib24}, on every level $n$ the resulting Galois groups $G = Gal(X^n+c)$ over $\mathbb{Q}$ for different $c$ are all (full) subgroups of one and the same bigger group $B_n$, of order $n \phi(n)$, with:
\begin{equation}
    G < B_n := Aut(D_{2n}) \cong C_n \rtimes Aut(C_n)
    = Hol(C_n) \cong C_n \rtimes (C_n)^\times
    \nonumber
\end{equation}
 
For $n=8$, we thus obtain for the Pauli group $P$: 
\begin{equation}
    P < B_8 = C_8 \rtimes V_4 = P \rtimes_2 C_2 = G_{32}^{(43)}
    \nonumber
\end{equation}
with this latter group, of order 32 and type 43 \cite{bib5}, as Galois group of $X^8+c$ for e.g. $c=3$. 
Altogether, for $n=8$, we obtain  the following complete list of Galois groups over $\mathbb{Q}^*$ for 
\textit{irreducible} $f_c=X^8+c$, $c\in \mathbb{Q}$:\begin{equation}
    c=d^4: Gal(f_c) = K_8 \nonumber
\end{equation}
\begin{equation}
    c=+2d^2: Gal(f_c) = D_{16} \nonumber
\end{equation}
\begin{equation}
    c=-2d^2: Gal(f_c) = QD_{16} \nonumber
\end{equation}
\begin{equation}
    c=k^2\ne d^4, 4d^4: Gal(f_k) = P \nonumber
\end{equation}
\begin{equation}
\text{else:      } Gal(f_c) = B_8 = C_8 \rtimes V_4 = P \rtimes_2 C_2 \nonumber 
\end{equation}
Thus we see that in the above sense the Pauli group can be uniquely characterised: for $n=2^s$, it is the smallest \textit{non-abelian} group that occurs as a \textit{proper} subgroup of a $B_n=Hol(C_n)$ in its role as Galois group of a pure polynomial, when $X^n+c$ does not any more fulfil the Schinzel condition for abelian Galois groups, for a given fixed $c$, and $Gal(X^{n/2}+c)$ abelian (where $s>2$, removing the trivial irreducible case $n=4$: all $Gal(X^{n/2}+c) \cong C_2$ for any $c\ne -d^2$. Here the non-abelian $D_8=Dih_4$ occurs as Galois group of all  $X^4+c, c\ne \pm d^2$ and is \textit{identical} to $B_4=Hol(C_4)= C_4 \rtimes C_2$). \newline 
For $s=3, n=8: c=k^2\ne d^4; k,d\in \mathbb{Q}^*$ always leads to the Galois group $P$, whereas  for $s=4, n=16: c=k^4\ne d^8$ gives rise to different Galois groups: 
to $M_4(2)$, e.g. for $c=2^4,2^{12}$, or to the order 32 group $G_{32}^{(38)} = C_8 \circ P$, e.g. for $c=3^4,5^4$ \cite{bib25}. 

\section{$P$-related algebras and quadratic forms}\label{sec5}

Our focus has been to determine the simplest possible 
polynomials in $\mathbb{Q}[X]$ (of given minimal degree 8) with Galois group $P$, thus providing an explicit and  constructive  realisation, and not just an existence proof, of a solution of the inverse Galois problem in this case. We were able to prove that the Pauli group $P$, unlike $Q_8$, can be realised by certain irreducible pure polynomials of the form $X^8+k^2$ over $\mathbb{Q}$. \newline 

We would like to set this into another context by noting that, during the last decades, several authors, including Grundman, Minác, Smith et al \cite{bib11},\cite{bib15}, and Jensen, Ledet, Yui \cite{bib10} have addressed
(i) the realizability, and 
(ii) an "extended" realisation that lends itself to a constructive solution of the inverse Galois problem,
respectively, of the Pauli group (and certain other groups)
as Galois extension over any field $K$ with characteristic not equal to $2$.
Note that Jensen et al use $QC$ for the Pauli group, alluding to 
$P\cong Q_8 \circ C_4 \cong Q\circ C \cong D_8\circ C_4$ (whereas Grundman et al call it $DC$). They obtained the following results:\newline 

ad (i): The realizability over $K$ of these groups, and of the Pauli group $P$ in particular, as a Galois group can be characterized by a relation between 
quaternion algebras in $Br(K)$: 

A necessary  and sufficient condition for the realizability of $P$ over $K$ is 
the existence of three linearly independent $A,B,C \in K^*/(K^*)^2$ such that in $Br(K)$:
\begin{equation}
    (A,B)(C,-1) = 1 \nonumber
\end{equation}
This $(A,B)=(C,C)$ is also equivalent to a manifestly $A',B',C'$-symmetric condition in $Br(K)$:
\begin{equation}
    (A',B')(A',C')(B',C')= 1 \nonumber
\end{equation}
 via the substitution $A'=AC, B'=BC, C'=C$ \cite{bib15}. By another possible choice, $A"=A, B"=AB, C"=C$, one obtains the (even $A",B"$-asymmetric) equivalent condition 
 $(A",A"B") = (C",C")$
in $Br(K)$, used by Ledet and coauthors. Here, as an example and a hint for later caution, the triplet $A"=k, B"=-1, C"=+2$ satisfies this latter condition over $K=\mathbb{Q}$, proving the realizability of a $P$-extension  $E/\mathbb{Q}$ that embeds 
our given 
$L=\mathbb{Q}(\sqrt{k},i,\sqrt{2}), [E:L]=2$, while none of the members of the triplet (let alone $C"=+2$) gives rise to a $Q_8$-extension (E/$\mathbb{Q}(\sqrt{C"})$).
\newline 
Equivalently, now staying with our case $K=\mathbb{Q}$, we state that
a necessary  and sufficient condition is to require:
$(a,a)(b,b)(c,c)(a,b)= 1$ in $Br(\mathbb{Q})$,
where again the triquadratic $E_8$-extension $L/\mathbb{Q}=\mathbb{Q}(\sqrt{a},\sqrt{b},\sqrt{c})/\mathbb{Q}$ embeds into a Pauli-extension $E/\mathbb{Q}$, with 
three quadratically independent $a,b,c \in \mathbb{Q}^*$.
By the standard rules \cite{bib17}, this is  equivalent to:
\begin{equation}
(abc,-1) = (a,b)
\end{equation}
in $Br(\mathbb{Q})$,
which in turn is equivalent to the condition $(A,B)=(C,C)$ above, by substituting $C=abc,A=a, B=b$, ie. by one of the choices that can freely be made, as long as $C$ need not be the fixed root that leads to a $Q_8$ extension $E/\mathbb{Q}(\sqrt{C})$.\newline

As is well known, any condition like (14) in $Br(\mathbb{Q})$ on two quaternion algebras $(u,v), (r,s)$ is equivalent to one on corresponding non-degenerate ternary quadratic forms over the same ground field, since the former two are algebra-isomorphic if and only if the latter two are similar \cite{bib16}:
\[(u,v) = (r,s) \Leftrightarrow  [-u,-v, uv]\cong [-r,-s,rs]\]

and thus (14) can be expressed
(since $(abc,-1) = (-a,-1)(+b,-1)(-c,-1) = (a,b)= (-a,b)(-1,b) = (-a,-b)(-a,-1)(-1,b)$, and hence after eliminating the last two algebras: $(-c,-1)=(-a,-b)$)
as the  equivalence over $\mathbb{Q}$ of the two forms
\begin{equation}
 [a,b,ab]\cong [1,c,c]
\end{equation}

which is the condition given in \cite{bib10} by Jensen et al.: 
If and only if a suitable triplet $a,b,c$ can be found that generates a given triquadratic extension $L/\mathbb{Q}$ and satisfies (15), 
$P$ can be realized over $\mathbb{Q}$ and embeds $L$.
\newline 
The left side of this constraint on $[a,b,ab]$ echoes the Witt-condition \cite{bib7}
\[(a,a)(a,b)(b,b)=1 \Leftrightarrow [a,b,ab]\cong [1,1,1]\]
for the embedability of a $V_4$- into a $Q_8$-extension (equivalences over $\mathbb{Q}(\sqrt{c})$ here!). Not every choice for (or permutation of) $a,b,c$ will necessarily satisfy this condition and (15) simultaneously, unless it is $c$ that leads to a $Q_8$-extension over $\mathbb{Q}(\sqrt{c})$.  \newline 

ad (ii): For an "extended" realisation of $P$,
in \cite{bib10} the additional claim is made that $E/\mathbb{Q}(\sqrt{c})$ be the $Q_8$ quaternion extension which embeds into the Pauli extension $P\cong Gal(E/\mathbb{Q})$ above, while $Gal(L/\mathbb{Q}) \cong E_8 = C_2^3$.
This introduces an explicit asymmetry between  the roles of $a,b,c$ 
(while possibly retaining symmetry between $a,b$). This explicit asymmetry is different in character to the one due to the "hidden $A',B',C'$-symmetry", hidden by the chosen substitutions, that we saw in the many equivalent  conditions for the "mere" realizability in (i).
\newline

Theorem 6.2.1 on p.143 in \cite{bib10} for $K=\mathbb{Q}$ states that condition (15) 
is necessary and sufficient
for the embedability of an $E_8$-extension $L/\mathbb{Q}=\mathbb{Q}(\sqrt{a},\sqrt{b},\sqrt{c})/\mathbb{Q}$ into a Pauli-extension $E/\mathbb{Q}$ such  that $E/\mathbb{Q}(\sqrt{c})$ is a $Q_8$ quaternion extension.
Now the hidden symmetry between $c$ and $a,b$ in (15) 
might seem to take care of this new explicit asymmetry, 
but in fact it does not.

Thus we would like to alert the reader to a subtle point: 
the characterisation (15) is formally only correct if one attaches to it an additional interpretation. The need for it can be proved as follows:
the choice of $a,b,c$ for 
$L/\mathbb{Q}=\mathbb{Q}(\sqrt{a},\sqrt{b},\sqrt{c})/\mathbb{Q}$ being generated by  
$a=-1,b=k, c=-2$ leads to two isotropic forms satisfying (15), whereas the equally possible choice
 $a'=2,b=k,c=-2$, giving the identical 
 $L =\mathbb{Q}(i,\sqrt{k},ir) = \mathbb{Q}(r,\sqrt{k},ir)$
 leads to a positive definite form $[a',b,a'b]$ and thus cannot possibly satisfy (15) ie. be similar to the isotropic form $[1,-2,-2]$ over $\mathbb{Q}$, albeit the fact that the correct choice of $c=-2$ does give rise to $Q_8$.  \newline
Hence a required interpretation of (15), layered into two steps, is the following:

1. The freedom to choose a triplet $a,b,c$ to yield a given $L/\mathbb{Q}$ cannot prevail if one wants to apply (15) in the context (ii) of theorem 6.2.1. Rather, it says only that, given a triquadratic extension $L/\mathbb{Q}$, if some three quadratically independent $a,b,c \in \mathbb{Q}^*$ can be found that yield  the $E_8$-extension $L/\mathbb{Q}$ and satisfy (15), then this is sufficient to ensure the realizability of $P$, and hence also the embedability of its normal subgroup 
$Q_8\cong Gal(E/R)$ for some quadratic extension $R/\mathbb{Q}$
into this Pauli extension $E/\mathbb{Q}$, via the given $L$.

2. By no means is it guaranteed, with  $c$ determined by $R=\mathbb{Q}(\sqrt{c})$ as fixfield to $Q_8$, that all choices of $a,b$ congruent with the given $L/\mathbb{Q}$ lead to (15) being satisfied, even while $E/R$ is simultaneously a $Q_8$ quaternion extension over $R$ (as the Witt condition after (15) might hold over $R$, e.g. $[a'=2,k,2k]\cong [1,1,1]$ over $\mathbb{Q}(\sqrt{-2})$, even if (15) fails over $\mathbb{Q}$). 
It is only possible to state that in $L$\textbackslash $\mathbb{Q}$ we can find  factors among the triplet $\sqrt{a},\sqrt{b},\sqrt{c},$ by which the radicals given by $a,b$ can be multiplied, so that the new triplet (modulo squares in $\mathbb{Q}$) does satisfy (15), which thus reveals itself as a sufficient but not a neccesary condition for the "extended" embedding.\newline

It may be possible to recover a full equivalence, given $L/\mathbb{Q}$, by defining the set $S_L=\{a,b,ab,c,ac,bc,abc\}$. 
Then the existence of at least one triplet $(u,v,x)$ of pairwise different, quadratically independent elements drawn from $S_L$, that satisfies (15), $[u,v,uv] \cong [1,x,x]$, seems necessary and sufficient for the "extended" embedding of $L$ into $E/\mathbb{Q}$, with  $Q_8\cong Gal(E/\mathbb{Q}(\sqrt{x}))$ and  $P\cong Gal(E/\mathbb{Q})$.\newline
 
Thus prepared, one can turn to the challenge of constructing explicit 
"generic" polynomials 
over a function field 
$K(t)$ with $char K\neq 2$,
which give rise to a given Galois extension.
This has been achieved, among other cases for small degree $n$, for the Pauli group in \cite{bib10}.
Our family of pure polynomials $f_k$ over $K=\mathbb{Q}$ then can  be seen as equivalent to a specialisation of some generic polynomial over $K(t)$. However, the fact that a polynomial as simple as the pure $f_k = X^8 + k^2$ can result is not visible from the complicated parametric form of the generic polynomials, usually given in powers of $X^2-d$ for some $d\in \mathbb{Q}^*$. \newline

Another approach \cite{bib19} via polynomials in $\mathbb{Q}(t)[X]$ yields for the Pauli group:\newline
$f_{8T11}(t,X)
= X^8+8X^6+4(4t^2-11)X^4+8(t^2-3)(t^2-2)X^2+t^2(t^2-3)^2$, with $t\in \mathbb{Q}^*$, $P=8T11$,
none of which specialises to a pure polynomial. 
Over $\mathbb{Q}$, Malle and Matzat give as an example polynomial:
$f_{8T11}(X) = X^8-X^5-2X^4+4X^2+X+1$.
\newline

On the other hand, it was recently shown by Chen, Chin, Tan \cite{bib21}
with an elaborate method using linear resolvents, Mathematica and an examination of all conjugacy classes of subgroups of $S_8$, 
that the Galois groups over $\mathbb{Q}$ of doubly even octic polynomials $X^8 + aX^4 + b, b=k^2$ a square in $\mathbb{Q}$, are of one of six isomorphism types:

$K_8, E_8, D_8, C_2 \times D_8, P=8T11 \text{  and the extraspecial group  } G_{32}^{(49)}= 2_+^{1+4}$, \newline
ie. among them $K_8$ and $P$ that we found for certain cases with $a=0$
(but not $Q_8$, see our earlier remark in section 4). 
$P$ occurs for certain conditions on $a,k\in \mathbb{Q}, 
k \notin \mathbb{Q}^2$ that include (1) in the case $a=0$. However, the case of non-square $b$ was not in the scope of their investigation, which is complementary to ours into pure polnomials with any constant term. \newpage





\section{Appendix: Pauli group definitions}\label{sec6}
As generators of the rank $r=3$ group $P$, the Pauli  matrices $X,Y,Z$ give the equivalent group presentations (with neutral element $E$): 
\begin{equation}
P=\langle A,X,C | A^4=X^2=E, XAX=A^{-1}, A^2=C^2, AC=CA, XC=CX \rangle \nonumber \end{equation}
\begin{equation}
=\langle A,J,C | A^2=J^2=C^2=-E, JA=-AJ, AC=CA, JC=CJ \rangle \nonumber
\end{equation}

\begin{equation}
=\langle X,Y,Z | X^2=Y^2=Z^2 = E, (YZ)^4=(ZX)^4=(XY)^4=E \rangle \nonumber
\end{equation} \newline
where $A = ZX = iY$, $C = XYZ = iE$ (called the Chiral element, with $C^2=-E$ as central involution), $J:=CX=(XYZ)X=iX$ (compare $A^X = XAX=A^{-1}=A^3=-A$ to the semidirect product in chapter 2),
showing explicitly the relation to the common presentations of:
\begin{equation}
 D_8 = \langle A,X | A^4=X^2=E, XAX=A^{-1}\rangle  =\langle iY,X\rangle\nonumber 
\end{equation}
\begin{equation}
 Q_8 = \langle A,J | A^4=E, A^2=J^2, JAJ^{-1}=A^{-1}\rangle =\langle iY,iX\rangle \nonumber
\end{equation}\newline
 With $K=iZ$, the rank 2 group $Q_8$ can also be presented by three order-4 generator elements: 
 $Q_8 = \langle A,J,K |  A^2=J^2=K^2=AJK \rangle $. Here $A=iY$ was used instead of $I$ to show the relation to $D_8$ and to avoid confusion with the complex $i$ or identity $Id$. \newline
 
Elements of $P$ thus have orders 1,2 or 4 only, and its center $Z(P) = <C> \cong C_4$ consists of the   order-2 element $-E=[X,Y]=C^2$ and the two order-4 elements $C=XYZ=iE$ and $-C=[X,Y]XYZ=-iE$, correcting the $Z(P) \cong C_2$ in \cite{bib10}. 

All its non-trivial quotients are elementary abelian: $C_2^k=E_{2^k}, k=1,2,3$.
  
Its group commutators, e.g. $[X,Y] = XYX^{-1}Y^{-1} = -E$, have to be distinguished from those (in a $d=2$-dimensional representation $\rho$ over $\mathbb{R}$) of the (real) Lie-algebra $\rho(su_{N}(\mathbb{C})) < M_d(\mathbb{C}), N=2$ (where they anticommute): $[iX,iY]_{Lie} = -2iZ$, with the $r=N^2-1=3$ Pauli matrices (times $i$) spanning the real representation of $su_2(\mathbb{C})$. \newline

The Pauli group can be constructed as a group extension in one of several ways \cite{bib5}:\newline

- as a split extension
$1 \longrightarrow Q_8 \longrightarrow P \longrightarrow C_2 \longrightarrow 1$ of $Q_8$ acted upon non-trivially by $C_2$, 
ie. as the semidirect product described above,
$P = Q_8 \rtimes_2 C_2$. It shares this property with the quasidihedral group $QD_{16}$. However, the maps describing the non-trivial product are different. 
\newline

- as a central non-split (stem) extension 
$1 \longrightarrow C_2 \longrightarrow P \longrightarrow E_8 \longrightarrow 1$ 
of $C_2$ acted upon trivially by $E_8 = C_2^3$. 
This realisation lends itself to the Pauli group embedding problem in inverse Galois theory, as in section 5, where $P/C_2 \cong E_8$, and the triquadratic extension $L/\mathbb{Q}$ is given, to be embedded into $E/\mathbb{Q}$. \newline

- as the (unique) central product, ie. built on a common central "overlap" subgroup $C_2=Z(Q_8)$ between $C_4$ and $Q_8$ (alternatively between $C_4$ and $D_8 = C_4 \rtimes C_2$, see the comment on the presentations above). As such it is in fact the smallest non-trivial central product among the finite groups, and the only one with order less than 24:\newline 
$P\cong C_4 \circ Q_8 \cong (C_4 \times Q_8)/C_2  \cong C_4 \circ D_8$, called push-out  of $C_4$ and $D_8$  by S. Lang.

As a note of caution, we alert the reader to the fact that some investigations use the name (generalised)   Pauligroups, $P_d$, to describe discretisations of the Heisenberg-Weyl (Lie) group $H(\mathbb{R})$, not via the infinite ring $\mathbb{Z}$ but via the finite rings $\mathbb{Z}_d =\mathbb{Z}/d\mathbb{Z}$ with $2 \leq d \in \mathbb{N}$, including cases with prime (power) $d=p^r, r\geq 1$ and Galois fields $\mathbb{F}_{p^r}$. These $P_d$ are  finite non-abelian groups of order $d^3$, hence none of them can be the classical order-16 group of our concern. Quite on the contrary, the simplest variant, $P_2$ in $d=2$ is isomorphic to the order-8 group $D_8=C_4 \rtimes C_2$. 

Other authors use the name Pauli group $P'_2$ for the classical order-8 group of quaternions,  $Q_8 \cong \{\pm E, \pm iX, \pm iY, \pm iZ\}=\langle iX,iY\rangle$, and refer to the classical Pauli group $P$ as a doubling extension of $P'_2$, alluding to $P=Q_8 \rtimes _2 C_2$. Care has to be taken when constructing tensor products of these generalised Pauli groups, where the cases $char(\mathbb{Z}_d) =2$ have to be treated separately. Yet in other contexts we have seen the label Pauli group used for the order-4, Klein  group $V_4$, referring to the "classical $P$ modulo any global phase", obtained by  quotienting out the center $Z(P) \cong \langle iE \rangle \cong C_4$ to yield 
$P/Z(P) \cong V_4 = C_2^2 = E_4$, and in general an elementary abelian group $E_{d^2}$ of order $d^2$. 

\begin{thebibliography}{23}
\ifx \bisbn   \undefined \def \bisbn  #1{ISBN #1}\fi
\ifx \binits  \undefined \def \binits#1{#1}\fi
\ifx \bauthor  \undefined \def \bauthor#1{#1}\fi
\ifx \batitle  \undefined \def \batitle#1{#1}\fi
\ifx \bjtitle  \undefined \def \bjtitle#1{#1}\fi
\ifx \bvolume  \undefined \def \bvolume#1{\textbf{#1}}\fi
\ifx \byear  \undefined \def \byear#1{#1}\fi
\ifx \bissue  \undefined \def \bissue#1{#1}\fi
\ifx \bfpage  \undefined \def \bfpage#1{#1}\fi
\ifx \blpage  \undefined \def \blpage #1{#1}\fi
\ifx \burl  \undefined \def \burl#1{\textsf{#1}}\fi
\ifx \doiurl  \undefined \def \doiurl#1{\url{https://doi.org/#1}}\fi
\ifx \betal  \undefined \def \betal{\textit{et al.}}\fi
\ifx \binstitute  \undefined \def \binstitute#1{#1}\fi
\ifx \binstitutionaled  \undefined \def \binstitutionaled#1{#1}\fi
\ifx \bctitle  \undefined \def \bctitle#1{#1}\fi
\ifx \beditor  \undefined \def \beditor#1{#1}\fi
\ifx \bpublisher  \undefined \def \bpublisher#1{#1}\fi
\ifx \bbtitle  \undefined \def \bbtitle#1{#1}\fi
\ifx \bedition  \undefined \def \bedition#1{#1}\fi
\ifx \bseriesno  \undefined \def \bseriesno#1{#1}\fi
\ifx \blocation  \undefined \def \blocation#1{#1}\fi
\ifx \bsertitle  \undefined \def \bsertitle#1{#1}\fi
\ifx \bsnm \undefined \def \bsnm#1{#1}\fi
\ifx \bsuffix \undefined \def \bsuffix#1{#1}\fi
\ifx \bparticle \undefined \def \bparticle#1{#1}\fi
\ifx \barticle \undefined \def \barticle#1{#1}\fi
\bibcommenthead
\ifx \bconfdate \undefined \def \bconfdate #1{#1}\fi
\ifx \botherref \undefined \def \botherref #1{#1}\fi
\ifx \url \undefined \def \url#1{\textsf{#1}}\fi
\ifx \bchapter \undefined \def \bchapter#1{#1}\fi
\ifx \bbook \undefined \def \bbook#1{#1}\fi
\ifx \bcomment \undefined \def \bcomment#1{#1}\fi
\ifx \oauthor \undefined \def \oauthor#1{#1}\fi
\ifx \citeauthoryear \undefined \def \citeauthoryear#1{#1}\fi
\ifx \endbibitem  \undefined \def \endbibitem {}\fi
\ifx \bconflocation  \undefined \def \bconflocation#1{#1}\fi
\ifx \arxivurl  \undefined \def \arxivurl#1{\textsf{#1}}\fi
\csname PreBibitemsHook\endcsname

\bibitem[\protect\citeauthoryear{Socolovsky}{2004}]{bib20}
\begin{barticle}
\bauthor{\bsnm{Socolovsky}, \binits{M.}}:
\batitle{The $cpt$ group of the dirac field}.
\bjtitle{Int. J. Theor. Phys.}
\bvolume{43},
\bfpage{1941}
(\byear{2004})
\end{barticle}
\endbibitem

\bibitem[\protect\citeauthoryear{Nielsen and Chuang}{2010}]{bib1}
\begin{bbook}
\bauthor{\bsnm{Nielsen}, \binits{M.}},
\bauthor{\bsnm{Chuang}, \binits{I.}}:
\bbtitle{Quantum Computation and Quantum Information (2nd Ed.)}.
\bpublisher{Cambridge University Press},
\blocation{Cambridge}
(\byear{2010})
\end{bbook}
\endbibitem

\bibitem[\protect\citeauthoryear{Peck}{2010}]{bib2}
\begin{barticle}
\bauthor{\bsnm{Peck}, \binits{R.}}:
\batitle{Imaginary transformations}.
\bjtitle{Journal of Mathematics and Music}
\bvolume{4}(\bissue{3}),
\bfpage{157}--\blpage{171}
(\byear{2010})
\end{barticle}
\endbibitem

\bibitem[\protect\citeauthoryear{Aaronson and Gottesman}{2004}]{bib3}
\begin{barticle}
\bauthor{\bsnm{Aaronson}, \binits{S.}},
\bauthor{\bsnm{Gottesman}, \binits{D.}}:
\batitle{Improved simulation of stabilizer circuits, and references therein}.
\bjtitle{Phys. Rev. A}
\bvolume{70},
\bfpage{052328}
(\byear{2004})
\end{barticle}
\endbibitem

\bibitem[\protect\citeauthoryear{Planat and Jorrand}{2008}]{bib13}
\begin{barticle}
\bauthor{\bsnm{Planat}, \binits{M.}},
\bauthor{\bsnm{Jorrand}, \binits{P.}}:
\batitle{On group theory for quantum gates and quantum coherence}.
\bjtitle{Journal of Physics}
\bvolume{41}(\bissue{hal-00263678v2}),
\bfpage{182001}
(\byear{2008})
\end{barticle}
\endbibitem

\bibitem[\protect\citeauthoryear{Min\'ac and Smith}{1991}]{bib15}
\begin{barticle}
\bauthor{\bsnm{Min\'ac}, \binits{J.}},
\bauthor{\bsnm{Smith}, \binits{T.}}:
\batitle{A characterization of c-fields via galois groups}.
\bjtitle{Journal of Algebra}
\bvolume{137},
\bfpage{1}
(\byear{1991})
\end{barticle}
\endbibitem

\bibitem[\protect\citeauthoryear{Shafarevich}{1954}]{bib4}
\begin{barticle}
\bauthor{\bsnm{Shafarevich}, \binits{I.}}:
\batitle{Construction of fields of algebraic numbers with given solvable galois
  group}.
\bjtitle{Izv. Akad. Nauk USSR}
\bvolume{18}(\bissue{1}),
\bfpage{525}--\blpage{578}
(\byear{1954})
\end{barticle}
\endbibitem

\bibitem[\protect\citeauthoryear{GroupNames.org}{2022}]{bib5}
\begin{bbook}
\bauthor{\bsnm{GroupNames.org}}:
\bbtitle{Database of Groups of Order up to 500}.
\bpublisher{online source},
\blocation{https://people.maths.bris.ac.uk/~matyd/GroupNames/about.html}
(\byear{2022})
\end{bbook}
\endbibitem

\bibitem[\protect\citeauthoryear{Jensen et~al.}{2002}]{bib10}
\begin{bbook}
\bauthor{\bsnm{Jensen}, \binits{C.U.}},
\bauthor{\bsnm{Ledet}, \binits{A.}},
\bauthor{\bsnm{Yui}, \binits{N.}}:
\bbtitle{Generic Polynomials - Constructive Aspects of the Inverse Galois
  Problem}.
\bpublisher{Cambridge University Press/MSRI},
\blocation{Cambridge}
(\byear{2002}).
\bcomment{esp. p.143-146, and references therein}
\end{bbook}
\endbibitem

\bibitem[\protect\citeauthoryear{Grundman and Smith}{1996}]{bib11}
\begin{barticle}
\bauthor{\bsnm{Grundman}, \binits{G.}},
\bauthor{\bsnm{Smith}, \binits{T.}}:
\batitle{Automatic realizability of galois groups of order 16}.
\bjtitle{Proc. of the AMS}
\bvolume{124}(\bissue{9}),
\bfpage{2631}
(\byear{1996})
\end{barticle}
\endbibitem

\bibitem[\protect\citeauthoryear{Spearman and Williams}{2008}]{bib22}
\begin{barticle}
\bauthor{\bsnm{Spearman}, \binits{B.K.}},
\bauthor{\bsnm{Williams}, \binits{K.S.}}:
\batitle{The simplest d4 octics}.
\bjtitle{Int. J. of Algebra}
\bvolume{2}(\bissue{2}),
\bfpage{79}--\blpage{89}
(\byear{2008})
\end{barticle}
\endbibitem

\bibitem[\protect\citeauthoryear{Jacobson and Velez}{1990}]{bib24}
\begin{barticle}
\bauthor{\bsnm{Jacobson}, \binits{E.T.}},
\bauthor{\bsnm{Velez}, \binits{W.Y.}}:
\batitle{The galois group of a radical extension of the rationals}.
\bjtitle{Manuscr. Math.}
\bvolume{67},
\bfpage{271}--\blpage{284}
(\byear{1990})
\end{barticle}
\endbibitem

\bibitem[\protect\citeauthoryear{Lemmermeyer}{1997}]{bib23}
\begin{barticle}
\bauthor{\bsnm{Lemmermeyer}, \binits{F.}}:
\batitle{Unramified quaternion extensions of quadratic number fields}.
\bjtitle{J. des Nombres de Bordeaux}
\bvolume{9}(\bissue{1}),
\bfpage{51}--\blpage{68}
(\byear{1997})
\end{barticle}
\endbibitem

\bibitem[\protect\citeauthoryear{Lorenz}{2006}]{bib6}
\begin{bbook}
\bauthor{\bsnm{Lorenz}, \binits{F.}}:
\bbtitle{Algebra, Vol I: Fields and Galois Theory, Chapter 14, Theorem 2}.
\bpublisher{Springer},
\blocation{New York}
(\byear{2006})
\end{bbook}
\endbibitem

\bibitem[\protect\citeauthoryear{Rocchetto and Russo}{2019}]{bib12}
\begin{botherref}
\oauthor{\bsnm{Rocchetto}, \binits{A.}},
\oauthor{\bsnm{Russo}, \binits{F.G.}}:
Decomposition of Pauli groups via weak central products.
arXiv:1911.10158
(2019)
\end{botherref}
\endbibitem

\bibitem[\protect\citeauthoryear{Witt}{1936}]{bib7}
\begin{barticle}
\bauthor{\bsnm{Witt}, \binits{E.}}:
\batitle{Konstruktion von körpern zu vorgegebener gruppe der ordnung $p^f$}.
\bjtitle{[Crelle's] Journal f.d. reine und angew. Mathematik}
\bvolume{174}(\bissue{1}),
\bfpage{243}
(\byear{1936})
\end{barticle}
\endbibitem

\bibitem[\protect\citeauthoryear{Chen et~al.}{2022}]{bib21}
\begin{botherref}
\oauthor{\bsnm{Chen}, \binits{M.H.W.}},
\oauthor{\bsnm{Chin}, \binits{A.Y.M.}},
\oauthor{\bsnm{Tan}, \binits{T.S.}}:
Galois groups of certain even octic polynomials.
arXiv:2210.10257
(2022)
\end{botherref}
\endbibitem

\bibitem[\protect\citeauthoryear{Schinzel}{1977}]{bib8}
\begin{barticle}
\bauthor{\bsnm{Schinzel}, \binits{A.}}:
\batitle{Abelian binomials, power residues and exponential congruences}.
\bjtitle{Acta Arithmetica}
\bvolume{32}(\bissue{1}),
\bfpage{245}
(\byear{1977})
\end{barticle}
\endbibitem

\bibitem[\protect\citeauthoryear{Darbi}{1925}]{bib9}
\begin{barticle}
\bauthor{\bsnm{Darbi}, \binits{G.}}:
\batitle{Sulla reducibilità delle equazioni algebriche}.
\bjtitle{Annali di Mat. pura e appl.}
\bvolume{4}(\bissue{1}),
\bfpage{185}--\blpage{208}
(\byear{1925})
\end{barticle}
\endbibitem

\bibitem[\protect\citeauthoryear{PARI/GP}{2023}]{bib25}
\begin{bbook}
\bauthor{\bsnm{PARI/GP}}:
\bbtitle{Bill Allombert, Private Communication}.
\bpublisher{online source},
\blocation{https://pari.math.u-bordeaux.fr/gpwasm.html}
(\byear{2023})
\end{bbook}
\endbibitem

\bibitem[\protect\citeauthoryear{Lorenz}{2008}]{bib17}
\begin{bbook}
\bauthor{\bsnm{Lorenz}, \binits{F.}}:
\bbtitle{Algebra, Vol II: Fields with Structure, Algebras and Advanced Topics,
  Chapter 30, F5}.
\bpublisher{Springer},
\blocation{New York}
(\byear{2008})
\end{bbook}
\endbibitem

\bibitem[\protect\citeauthoryear{Voight}{2021}]{bib16}
\begin{bbook}
\bauthor{\bsnm{Voight}, \binits{J.}}:
\bbtitle{Quaternion Algebras}.
\bpublisher{Springer},
\blocation{Cham}
(\byear{2021}).
\bcomment{chap. 5.2}
\end{bbook}
\endbibitem

\bibitem[\protect\citeauthoryear{Malle and Matzat}{1999}]{bib19}
\begin{bbook}
\bauthor{\bsnm{Malle}, \binits{G.}},
\bauthor{\bsnm{Matzat}, \binits{B.H.}}:
\bbtitle{Inverse Galois Theory}.
\bpublisher{Springer},
\blocation{Berlin, Heidelberg, New York}
(\byear{1999})
\end{bbook}
\endbibitem

\end{thebibliography}

\end{document}